\documentclass[11pt,a4paper,leqno]{article}
\usepackage[utf8]{inputenc}

\usepackage{amsmath} 
\usepackage{amsthm} 
\usepackage{amssymb}	
\usepackage{graphicx}
\usepackage{eucal}
\usepackage{xcolor}  
\usepackage{stmaryrd}
\usepackage[shortlabels]{enumitem}
\usepackage{pdfpages}
\usepackage[textwidth=16cm,centering,headheight=24pt]{geometry} 
\usepackage{comment}
\usepackage{scalerel,stackengine}
\stackMath
\newcommand\reallywidehat[1]{%
\savestack{\tmpbox}{\stretchto{%
  \scaleto{%
    \scalerel*[\widthof{\ensuremath{#1}}]{\kern-.6pt\bigwedge\kern-.6pt}%
    {\rule[-\textheight/2]{1ex}{\textheight}}
  }{\textheight}%
}{0.5ex}}%
\stackon[1pt]{#1}{\tmpbox}%
}
\usepackage{url}

\newcommand{\abs}[1]{\left| #1 \right|} 
\newcommand{\norm}[1]{\left\| #1 \right\|}
 
\newcommand{\B}{\mathbb B}
\newcommand{\R}{\mathbb R}

\newcommand{\N}{\mathbb N}

\newcommand{\Z}{\mathbb Z}

\renewcommand{\H}{\mathbb{H}}

\renewcommand{\S}{\mathbb{S}}

\newcommand{\eps}{\varepsilon}

\DeclareMathOperator{\vol}{vol}

\DeclareMathOperator{\D}{D}
\let\d\relax
\DeclareMathOperator{\d}{d}

\DeclareMathOperator{\spt}{spt}

\DeclareMathOperator{\dist}{dist}

\DeclareMathOperator{\diam}{diam}
\DeclareMathOperator{\inj}{inj}

\DeclareMathOperator{\Riem}{Riem}

\DeclareMathOperator{\tr}{tr}
\DeclareMathOperator{\reach}{reach}

\newtheorem{theorem}{Theorem}
\newtheorem*{theorem*}{Theorem}

\newtheorem{prop}[theorem]{Proposition}

\newtheorem{mainthm}{Theorem}           

\theoremstyle{definition}
\newtheorem{definition}{Definition}

\newtheorem{example}[theorem]{Example}

\theoremstyle{remark}
\newtheorem{remark}{Remark}
\newtheorem*{remark*}{Remark}
\newtheorem*{notation}{Notation}

\title{Singular extension of critical Sobolev mappings with values into 
complete Riemannian manifolds}

\date{}
\author{Federico Luigi Dipasquale ($*$)}

\newcommand{\Addresses}{{
  \bigskip
  \footnotesize

  Federico Luigi Dipasquale (Corresponding author), \textsc{Scuola Superiore Meridionale, via Mezzocannone, 4 80138 Napoli (Italy)}\par\nopagebreak
  \textit{E-mail address}: \texttt{f.dipasquale@ssmeridionale.it}
}}

\begin{document}

\maketitle

\begin{abstract}
   Relying on a recent criterion, due to A.~Petrunin \cite{P}, 
   to check if a complete, non-compact, Riemannian 
   manifold admits an isometric embedding with 
   positive reach into a Euclidean space, 
   we extend to manifolds with such a property
   the singular extension results of B.~Bulanyi and 
   J.~Van~Schaftingen \cite{BVS} for maps in the critical, nonlinear Sobolev space 
   $W^{m/(m+1),m+1}\left(X^m,\mathcal{N}\right)$, 
   where $m \in \N \setminus \{0\}$, $\mathcal{N}$ is a compact Riemannian 
   manifold, and $X^m$ is either the sphere $\S^m = \partial \B^{m+1}_1$, 
   the plane $\R^m$, or 
   again $\S^m$ but seen as the boundary sphere of the Poincar\'{e} ball model 
   of the hyperbolic space $\H^{m+1}$. 
   As in \cite{BVS}, we obtain that the extended maps satisfy an exponential 
   weak-type Sobolev-Marcinkiewicz estimate. Finally, we provide some 
   illustrative examples of non-compact target manifolds to which our 
   results apply.
\end{abstract}

\section*{Introduction}

The purpose of this note is to point out explicitly some direct consequences 
of the proof of the singular extensions theorems in \cite{BVS}, 
as well as their connection with
a recent, sufficient criterion \cite{P} to 
establish if a complete Riemannian manifold $\mathcal{N}$ admits an 
isometric Euclidean embedding with positive reach. 
We recall that the {\em reach} of a submanifold $\mathcal{S}$ of $\R^\nu$ is the 
supremum of the sizes of the tubular neighbourhoods of $\mathcal{S}$ 
on which the nearest-point projection onto $\mathcal{S}$ is well defined. 
In the following, we will denote $\reach{\mathcal{N}}$ the reach of a 
Euclidean isometric embedding of $\mathcal{N}$.

To see in which way the notion of reach is connected with the results 
in~\cite{BVS}, we recall that in~\cite[Theorem~1.1]{BVS}, B.~Bulanyi and 
J.~Van~Schaftingen proved that, if 
$\mathcal{N}$ is \emph{any} compact, smooth Riemannian manifold 
isometrically embedded in 
$\R^\nu$, for some $\nu \in \N$, via Nash's theorem, then any map $u$ in the 
critical, nonlinear Sobolev space 
$W^{m/(m+1),m+1}\left(\mathbb{S}^m, \mathcal{N}\right)$ can be extended to a map 
$U \in W^{1,1}\left(\mathbb{B}^{m+1}_1, \mathcal{N}\right)$ 
whose trace on $\mathbb{S}^m$ is $u$ and satisfying the following 
weak-type exponential Marcinkiewicz-Sobolev estimate:  
for every $t \in (0,\infty)$,
\begin{multline}\label{eq:wSM-est}
	t^{m+1} \mathcal{L}^{m+1}\left(\left\{ \left. x \in \mathbb{B}^{m+1}_1 \,\right\vert \, \abs{{\D}U(x)} \geq t \right\} \right) \\
	\leq A \exp \Biggl( B \iint \limits_{\substack{ (x,y) \in \S^m \times \S^m \\ d(u(x),u(y)) \geq \delta}} \frac{1}{\abs{x-y}^{2m}} \,{\d}x\,{\d}y \Biggr) \iint \limits_{\S^m \times \S^m} \frac{d(u(x),u(y))^{m+1}}{\abs{x-y}^{2m}}\,{\d}x\,{\d}y.
\end{multline}
Here, $A$ is a constant depending only on $m$, $B$ depends only on $m$ and on the ratio 
$\diam \mathcal{N} / \reach \mathcal{N}$, and $\delta$ depends only on $m$ and on $\reach \mathcal{N}$ 
(in fact, it is the ratio of a dimensional constant and $\reach \mathcal{N}$; 
we recall that, for compact $\mathcal{N}$, 
$\reach \mathcal{N}$ is strictly positive, see, e.g., \cite[Section~4]{Federer}). 
The estimate~\eqref{eq:wSM-est} is then extended to the case in which 
$\S^m$ is replaced by $\R^m$ and $\B^{m+1}_1$ by $\R^{m+1}_+$ and to the case 
in which $\S^m$ is seen as the boundary sphere\footnote{That is to say, the 
set of points at infinity of the Poincar\'e ball model, in the sense explained, for instance, in \cite[Section~3.L]{GHLf}.} of the Poincar\'{e} ball model
of the hyperbolic space $\H^{m+1}$. 
(See, respectively, Theorem~1.2 and Theorem~1.3 in \cite{BVS}.)

We recall that, for $\Omega \subset \R^{m+1}$ an open set and for a 
Riemannian manifold $\mathcal{N}$, viewed as 
a submanifold of $\R^\nu$ thanks to Nash's theorem, the space 
$W^{1,p}\left(\Omega, \mathcal{N} \right)$ is usually defined as  
\[
		W^{1,p}\left( \Omega, \mathcal{N} \right) := 
		\left\{ U \in W^{1,p}(\Omega,\R^\nu) \,:\, U(x) \in \mathcal{N} \,\,\mbox{for a.e. } x \in \Omega  \right\}.
\]
In case $\mathcal{N}$ is complete, it is known that this definition does 
not depend on the choice of the isometric Euclidean embedding 
(e.g., \cite[Proposition~2.1]{BousquetPonceVanSchaftingen}). 
In particular, one immediately sees this is the case if 
$\mathcal{N}$ admits an isometric Euclidean embedding with positive reach.  
Indeed, in such a case $\mathcal{N}$ is closed as a subset of $\R^\nu$ 
(c.f.~\cite[Remark~4.2]{Federer}), 
and, consequently, it is complete as a Riemannian manifold. (To see this, 
observe that 
any Cauchy sequence in $\mathcal{N}$ is bounded with respect to the 
geodesic distance and that, since the embedding is isometric, 
locally around any point of $\mathcal{N}$, the geodesic distance and the 
Euclidean distance are comparable.)
We define the space $W^{s,p}(\partial \Omega,\mathcal{N})$ as
\[
	W^{s,p}\left(\partial \Omega,\mathcal{N}\right) := 
	\left\{ u \in L^p\left(\partial \Omega, \R^\nu\right) \,:\, u(x) \in \mathcal{N} \mbox{ a.e. in } \partial \Omega \mbox{ and } \iint\limits_{\partial \Omega \times \partial \Omega} \frac{d(u(x),u(y))^{p}}{\abs{x-y}^{sp+m}}\,{\d}x\,{\d}y < \infty \right\},
\]
where $d(\cdot,\cdot) : \mathcal{N} \times \mathcal{N} \to \R_+$ denotes 
the geodesic distance in $\mathcal{N}$. This definition is independent of 
the choice of the isometric embedding of $\mathcal{N}$ into $\R^\nu$ 
(independently of the compactness of $\mathcal{N}$). 
A careful reading of the proofs of the results in \cite{BVS} shows that 
the compactness of $\mathcal{N}$ is used only for two purposes: {($i$)} to 
ensure that the image of the Nash isometric embedding of $\mathcal{N}$ into 
$\R^\nu$ has positive reach, with well-defined and smooth nearest-point 
projection, and {($ii$)} to ensure that the so-called \emph{gap potential}
\[
\iint \limits_{{\substack{ (x,y) \in \S^m \times \S^m \\ d(u(x),u(y)) \geq \delta}}} \frac{1}{\abs{x-y}^{2m}} \,{\d}x\,{\d}y
\]
controls the truncated Gagliardo energy of $u$. More precisely, 
for any $\delta > 0$, the gap 
potential trivially satisfies the inequality 
\[
\iint \limits_{{\substack{ (x,y) \in \S^m \times \S^m \\ d(u(x),u(y)) \geq \delta}}} \frac{1}{\abs{x-y}^{2m}} \,{\d}x\,{\d}y 
\leq
\frac{1}{\delta^{m+1}}\iint \limits_{{\substack{ (x,y) \in \S^m \times \S^m \\ d(u(x),u(y)) \geq \delta}}} \frac{d(u(x),u(y))^{m+1}}{\abs{x-y}^{2m}} \,{\d}x\,{\d}y.
\]
In addition, if $\diam_{\mathcal{N}}(u(\S^m))$ is bounded 
(for instance, if $\mathcal{N}$ is compact), then 
\[
\iint \limits_{{\substack{ (x,y) \in \S^m \times \S^m \\ d(u(x),u(y)) \geq \delta}}} \frac{d(u(x),u(y))^{m+1}}{\abs{x-y}^{2m}} \,{\d}x\,{\d}y \\
\leq 
\diam_{\mathcal{N}}(u(\S^m))^{m+1}
\iint \limits_{{\substack{ (x,y) \in \S^m \times \S^m \\ d(u(x),u(y)) \geq \delta}}} \frac{1}{\abs{x-y}^{2m}} \,{\d}x\,{\d}y. 
\]
From these inequalities and a careful reading of the proofs, one immediately 
deduces that \cite[Theorems~1.1, 1.2, and 1.3]{BVS} still hold if 
$\mathcal{N}$ is any Riemannian manifold admitting an isometric Euclidean embedding 
with positive reach and if one restricts to maps $u$ whose range has bounded diameter in $\mathcal{N}$. 
   
In this note, we put the observations above in precise and slightly more 
general terms. Our main results are the following theorems, which are direct extensions of \cite[Theorem~1.1, 1.2, and 1.3]{BVS}. In those, we always assume that $\mathcal{N}$ is connected  
and we notice that, as observed in \cite[p.~284]{M-VS1}, this is 
a harmless condition in the extension problem considered in this paper, 
because by \cite[Theorem~6.5]{BrezisMironescu} it follows that 
the essential range of any function $V \in W^{1,1}_{\rm loc}(\Omega, \R^\nu)$ is connected 
whenever $\Omega \subset \R^{m+1}$ is a connected, open set or  
$\Omega = \overline{\R^{m+1}_+}$.

\begin{theorem}\label{thm:I}
   Let $m \in \N \setminus \{0\}$ and let $\mathcal{N}$ be a connected  
   Riemannian 
   manifold admitting an isometric Euclidean embedding with positive 
   reach. 
   There exists constants $A$, $B$, 
   $\delta \in (0,\infty)$, depending only on $m$ and the reach of 
   $\mathcal{N}$ (and further specified in~\eqref{eq:constants} below), 
   such that for every 
   $u \in W^{m/(m+1),m+1}\left(\S^m, \mathcal{N}\right)$ there exists a mapping 
   $U \in W^{1,1}\left(\B^{m+1}_1, \mathcal{N}\right)$ such that 
   $\tr_{\S^m} U = u$ and for every $t \in (0,\infty)$
   \begin{multline}\label{eq:I-1}
   	t^{m+1} \mathcal{L}^{m+1}\left(\left\{ \left. x \in \mathbb{B}^{m+1}_1 \,\right\vert \, \abs{{\D}U(x)} \geq t \right\} \right) \\
	\leq A \exp \Biggl( B \iint \limits_{\S^m \times \S^m} \frac{d(u(x),u(y))^{m+1}}{\abs{x-y}^{2m}} \,{\d}x\,{\d}y \Biggr) \iint \limits_{\S^m \times \S^m} \frac{d(u(x),u(y))^{m+1}}{\abs{x-y}^{2m}}\,{\d}x\,{\d}y,
   \end{multline}
   If, in addition, $\diam_{\mathcal{N}}{u(\S^m)} \leq L$, then
   \begin{multline}\label{eq:I-2}
	t^{m+1} \mathcal{L}^{m+1}\left(\left\{ \left. x \in \mathbb{B}^{m+1}_1 \,\right\vert \, \abs{{\D}U(x)} \geq t \right\} \right) \\
	\leq A \exp \Biggl( B' \iint \limits_{\substack{ (x,y) \in \S^m \times \S^m \\ d(u(x),u(y)) \geq \delta}} \frac{1}{\abs{x-y}^{2m}} \,{\d}x\,{\d}y \Biggr) \iint \limits_{\S^m \times \S^m} \frac{d(u(x),u(y))^{m+1}}{\abs{x-y}^{2m}}\,{\d}x\,{\d}y,   
   \end{multline}
   where $B' = L^{m+1} B$ depends only on $m$, $L$, and $\reach \mathcal{N}$. 
   Moreover, one can take 
   $U \in C\left( \B^{m+1}_1 \setminus S, \mathcal{N} \right)$, where the singular set 
   $S \subset \B^{m+1}_1$ is a finite set whose cardinality is controlled by the 
   right-hand side of~\eqref{eq:I-1} or, in case $\diam_{\mathcal{N}}{u(\S^m)} \leq L$, 
   by the right-hand side of~\eqref{eq:I-2}.
\end{theorem}

More precisely, the proof shows that
\begin{equation}\label{eq:constants}
	A = A(m), \qquad B = \frac{C(m)}{(\reach{\mathcal{N}})^{m+1}}, \qquad
	B' = C(m) \left(\frac{2KL}{\reach{\mathcal{N}}} \right)^{m+1},  
\end{equation}
where $C(m) \in (0,\infty)$ depends only on $m$.
In the particular case $\mathcal{N}$ is compact, $\diam \mathcal{N}$ 
is finite and we can meaningfully bound $d(u(x),u(y))$ in the 
exponential in~\eqref{eq:I-1} with $\diam \mathcal{N}$. 
Consequently,~\eqref{eq:I-2} reduces to 
\cite[Eq.~(5)]{BVS}, with $B'$ depending only on $m$ and the ratio 
$\diam \mathcal{N} / \reach \mathcal{N}$, exactly as in \cite{BVS}. 

\begin{theorem}\label{thm:II}
   Let $m \in \N \setminus \{0\}$ and let $\mathcal{N}$ be a connected
   Riemannian 
   manifold admitting an isometric Euclidean embedding with positive 
   reach. 
   There exists constants $A$, $B$, 
   $\delta \in (0,\infty)$, 
   depending only on $m$ and the reach of $\mathcal{N}$ and given by~\eqref{eq:constants}, 
   such that for every 
   $u \in W^{m/(m+1),m+1}\left(\R^m, \mathcal{N}\right)$ there exists a mapping 
   $U \in W^{1,1}_{\rm loc}\left(\overline{\R^{m+1}_+}, \mathcal{N}\right)$ such that 
   $\tr_{\R^m} U = u$ and for every $t \in (0,\infty)$
   \begin{multline}\label{eq:II-1}
   	t^{m+1} \mathcal{L}^{m+1}\left(\left\{ \left. x \in \R^m_+ \,\right\vert \, \abs{{\D}U(x)} \geq t \right\} \right) \\
	\leq A \exp \Biggl( B \iint \limits_{\R^m \times \R^m} \frac{d(u(x),u(y))^{m+1}}{\abs{x-y}^{2m}} \,{\d}x\,{\d}y \Biggr) \iint \limits_{\R^m \times \R^m} \frac{d(u(x),u(y))^{m+1}}{\abs{x-y}^{2m}}\,{\d}x\,{\d}y.
   \end{multline}
   If, in addition, $\diam_{\mathcal{N}}{u(\S^m)} \leq L$, then
   \begin{multline}\label{eq:II-2}
	t^{m+1} \mathcal{L}^{m+1}\left(\left\{ \left. x \in \mathbb{R}^{m+1}_+ \,\right\vert \, \abs{{\D}U(x)} \geq t \right\} \right) \\
	\leq A \exp \Biggl( B' \iint \limits_{\substack{ (x,y) \in \R^m \times \R^m \\ d(u(x),u(y)) \geq \delta}} \frac{1}{\abs{x-y}^{2m}} \,{\d}x\,{\d}y \Biggr) \iint \limits_{\R^m \times \R^m} \frac{d(u(x),u(y))^{m+1}}{\abs{x-y}^{2m}}\,{\d}x\,{\d}y,   
   \end{multline}
   where $B' = L^{m+1} B$ depends only on $m$, $L$, and $\reach \mathcal{N}$. 
 Moreover, one can take 
   $U \in C\left( \R^{m+1}_+ \setminus S, \mathcal{N} \right)$, where the singular set 
   $S \subset \R^{m+1}_+$ is a finite set whose cardinality is controlled by the 
   right-hand side of~\eqref{eq:II-1} or, in case $\diam_{\mathcal{N}}{u(\S^m)} \leq L$, 
   by the right-hand side of~\eqref{eq:II-2}.
\end{theorem}

As in \cite{BVS}, the symbol $W^{1,1}_{\rm loc}\left( \overline{\R^{m+1}_+}, \mathcal{N} \right)$
means that $U$ is weakly differentiable and that 
$\int_K \abs{{\D}U} < +\infty$ for every compact set 
$K \subset \overline{\R^{m+1}_+}$.

In our last theorem, we consider, as in \cite{BVS}, extensions of maps 
defined over $\S^m$ to $\B^{m+1}_1$, where we endow $\B^{m+1}_1$ with the 
Poincar\'e metric
\[
	g_{\rm hyp}(x) = \frac{4g_{\rm eucl}(x)}{\left(1-\abs{x}^2\right)^2}. 
\]
We recall that $\left(\B^{m+1}_1, g_{\rm hyp}\right)$ is a standard model 
(called the {\em Poincar\'{e} ball model}) 
for the hyperbolic space $\H^{m+1}$. In the statement below, $\mathcal{H}^{m+1}$ 
denotes the $(m+1)$-dimensional Hausdorff measure.
\begin{theorem}\label{thm:III}
   Let $m \in \N \setminus \{0\}$ and let $\mathcal{N}$ be a connected
   Riemannian 
   manifold admitting an isometric Euclidean embedding with positive reach. 
   There exists constants $A$, $B$, 
   $\delta \in (0,\infty)$, depending only on $m$ and the reach of $\mathcal{N}$ 
   and given by~\eqref{eq:constants}, 
   such that for every 
   $u \in W^{m/(m+1),m+1}\left(\S^m, \mathcal{N}\right)$ there exists a mapping 
   $U \in W^{1,1}\left(\B^{m+1}_1, \mathcal{N}\right)$ such that 
   $\tr_{\S^m} U = u$ and for every $t \in (0,\infty)$
   \begin{multline}\label{eq:III-1}
   	\mathcal{H}^{m+1}\left(\left\{ \left. x \in \H^{m+1} \,\right\vert \, \abs{{\D}U(x)} \geq t \right\} \right) \\
	\leq \frac{A}{t^{m+1}} \exp \Biggl( B \iint \limits_{\S^m \times \S^m} \frac{d(u(x),u(y))^{m+1}}{\abs{x-y}^{2m}} \,{\d}x\,{\d}y \Biggr) \iint \limits_{\S^m \times \S^m} \frac{d(u(x),u(y))^{m+1}}{\abs{x-y}^{2m}}\,{\d}x\,{\d}y.
   \end{multline}
   If, in addition, $\diam_{\mathcal{N}}{u(\S^m)} \leq L$, then
   \begin{multline}\label{eq:III-2}
	\mathcal{H}^{m+1}\left(\left\{ \left. x \in \mathbb{H}^{m+1} \,\right\vert \, \abs{{\D}U(x)} \geq t \right\} \right) \\
	\leq \frac{A}{t^{m+1}} \exp \Biggl( B' \iint \limits_{\substack{ (x,y) \in \S^m \times \S^m \\ d(u(x),u(y)) \geq \delta}} \frac{1}{\abs{x-y}^{2m}} \,{\d}x\,{\d}y \Biggr) \iint \limits_{\S^m \times \S^m} \frac{d(u(x),u(y))^{m+1}}{\abs{x-y}^{2m}}\,{\d}x\,{\d}y,   
   \end{multline}
   where $B' = L^{m+1} B$ depends only on $m$, $L$, and $\reach \mathcal{N}$. 
   Moreover, one can take 
   $U \in C\left( \H^{m+1} \setminus S, \mathcal{N} \right)$, where the singular set 
   $S \subset \H^{m+1}$ is a finite set whose cardinality is controlled by the 
   right-hand side of~\eqref{eq:III-1} or, in case $\diam_{\mathcal{N}}{u(\S^m)} \leq L$, 
   by the right-hand side of~\eqref{eq:III-2}.
\end{theorem}

Establishing if a Euclidean embedding of a non-compact Riemannian 
manifold has positive reach is notoriously a very difficult task, in general. 
Therefore, our results would be of very limited interest without 
a reasonable criterion for identifying Riemannian manifolds with this 
property. 
In this respect, A.~Petrunin proved in~\cite{P} the 
following theorem.
\begin{mainthm}[{Petrunin, \cite{P}}]\label{thm:p}
	Suppose $\mathcal{N}$ is a complete, smooth, connected Riemannian 
	manifold with $1$-bounded geometry. 
	Then, $\mathcal{N}$ admits an isometric tubed embedding into 
	a Euclidean space $\R^\nu$, for some $\nu \in \N$, if and only if 
	$\mathcal{N}$ has uniformly polynomial growth. The number $\nu$ can be 
	estimated in terms of the dimension $n$ of $\mathcal{N}$ and of the 
	degree of the growth polynomial of $\mathcal{N}$.
\end{mainthm}
In Theorem~\ref{thm:p}, a {\em tubed embedding} means a   
Euclidean embedding with positive reach, {\em $1$-bounded geometry} 
refers to the strict positivity of the injectivity radius along with the boundedness 
of the first covariant derivative of the Riemann curvature tensor, 
and {\em uniformly polynomial growth} to the fact that the volume of geodesic balls is 
controlled by a uniform polynomial function 
(called a {\em growth polynomial}) of the radius 
(precise definitions are given in Section~\ref{sec:embeddings}).
By results in \cite{LeobacherSteinicke}, if $\eps > 0$ 
is the reach 
of the tubed embedding of $\mathcal{N}$ into $\R^{\nu}$ provided by 
Theorem~\ref{thm:p}, then the nearest-point projection is smooth up to 
the boundary of each $\delta$-neighbourhood of $\mathcal{N}$, 
for every $\delta \in (0,\eps)$.

Although Theorem~\ref{thm:p} is not a characterisation of complete  
Riemannian manifolds admitting a tubed embedding, but only a 
sufficient criterion, 
as observed in \cite{P} the following conditions 
are instead necessary to this purpose: {($i$)} bounded sectional curvature; 
{($ii$)} positive injectivity radius; {($iii$)} uniformly polynomial growth. 
At present, it is unknown whether such conditions are also sufficient or not 
(see the discussion in \cite[§17]{P} for more information).

Thanks to Theorem~\ref{thm:p}, many complete, non-compact 
Riemannian manifolds are recognised as admitting tubed Euclidean embeddings. 
This allows us to apply our theorems in several situations not covered 
by the results in~\cite{BVS}.  
In the conclusive section of this paper, Section~\ref{sec:examples}, 
we discuss some examples. We emphasise in particular the following ones:
\begin{itemize}
	\item Warped products of the type $\R \times_f \mathcal{M}$, where 
	$\mathcal{M}$ is a compact Riemannian manifold and the warping function 
	$f$ satisfies the (mild) assumptions in Example~\ref{ex:warped}; 
	\item The universal covering of any compact Riemannian manifold whose 
	fundamental group has {\em polynomial growth} (see, e.g., \cite[Section~3.I]{GHLf}).
	This includes:
	\begin{itemize} 
		\item The universal covering of any compact Riemannian manifold with 
		Ricci curvature bounded below;
		\item The universal covering of any compact Riemannian manifold with 
		Abelian fundamental group.
	\end{itemize} 
\end{itemize}

In \cite{BVS}, the first main result that is proven is Theorem~1.2, from 
which Theorem~1.1 and Theorem~1.3 (essentially) follow by conformal 
parametrisation of 
the ball $\B^{m+1}_1$ (once endowed with the canonical metric and the other 
with the Poincar\'{e} metric) by $\R^{m+1}_+$. 
Quoting the discussion in \cite[p.~6]{BVS}, we observe that, although 
standard conformal transformations between $\B^{m+1}_1$, $\R^{m+1}_+$, 
and $\H^{m+1}$ preserve the Gagliardo energy of the trace in the right-hand sides 
of~\eqref{eq:I-2},~\eqref{eq:II-2}, and~\eqref{eq:III-2}, the 
corresponding gap potentials and the strong-type 
quantities $\int_{\B^{m+1}} \abs{{\D}U}^{m+1}$, $\int_{\R^{m+1}_+} \abs{{\D}U}^{m+1}$, 
$\int_{\H^{m+1}} \abs{{\D}U}^{m+1}$ corresponding to the weak-type 
quantities in their left-hand side, the left-hand sides as such are not conformally 
invariant, so \cite[Theorem~1.1, 1.2, and 1.3]{BVS}
are not equivalent to each other (and, therefore, the same is true for 
Theorem~\ref{thm:I},~\ref{thm:II}, and~\ref{thm:III} in the present paper). 
As commented in \cite{BVS}, this explains why one has  
the same construction of the extension in all the 
three cases, but three different particular estimates on $U$.
We follow the same line and, in Section~\ref{sec:proof}, we sketch the proof of 
Theorem~\ref{thm:II}. This will be enough because the differences with 
respect to \cite[Theorem~1.2]{BVS} are confined to the beginning of the 
proof, which for the rest follows exactly as in \cite{BVS}. 
With these modifications, Theorem~\ref{thm:I} and 
Theorem~\ref{thm:III} follow word-for-word as in \cite{BVS}, and therefore 
we will not repeat the arguments. To keep at 
minimum the size of this note, we always refer to the notation of 
\cite{BVS}, to which the reader is addressed for complete details. 
In the same spirit, we do not really try to explain the strategy of 
the proofs in \cite{BVS}, which is very neatly explained there. We only 
mention that the key point consists in dividing the domain appropriately, 
using $\lambda$-adic cubes, in regions where the {\em extension by averaging} 
of $u$ (see Section~\ref{sec:averaging}) is close (in a precise quantitative sense) 
to $\mathcal{N}$ and their complement. 
In the first, ``good'' regions, one can define the required extension by 
reprojecting the extension by averaging onto $\mathcal{N}$, with 
controlled energy, using the nearest-point projection. In the cubes in the 
complement of the good regions, one defines the extension by homogeneous 
extension of the trace on the boundary of the cubes (with respect to their 
barycentres). The fact that the number of ``bad 
cubes'' is exponentially bounded by the Gagliardo energy (or the gap 
potential, in the compact or bounded case) yields the ``bad'' exponential 
term in the Marcinkiewicz-Sobolev estimates above.

\numberwithin{equation}{section}
\numberwithin{definition}{section}
\numberwithin{theorem}{section}
\numberwithin{remark}{section}

\section{Positive reach and tubed embeddings}\label{sec:embeddings}
Following~\cite{P}, given $\mathcal{N}$ and $\mathcal{M}$ Riemannian 
manifolds, we say that that an embedding 
$f : \mathcal{N} \to \mathcal{M}$ 
is a {\em tubed embedding} if its image $f(\mathcal{N})$
has {\em positive reach} in $\mathcal{M}$. This implies that 
$f(\mathcal{N})$ is closed in $\mathcal{M}$ (in particular, a tubed 
embedding is a closed embedding) and,    
by definition, it means that there exists an $\eps$-neighbourhood 
$\mathcal{O}_\eps$ of 
$f(\mathcal{N})$ in $\mathcal{M}$, for some $\eps > 0$, on which 
the {\em nearest-point projection} 
$\Pi_{\mathcal{N}} : \mathcal{O}_\eps \to f(\mathcal{N})$, associating with each 
$z \in \mathcal{O}_\eps$ its closest point on $f(\mathcal{N})$, 
is well defined. In this note we are only interested in the case in which 
$\left(\mathcal{M}, g^\mathcal{M}\right) = \left(\R^\nu, {\rm eucl}\right)$.

\begin{notation}
Let $(\mathcal{N}, g)$ be a smooth Riemannian manifold of finite 
dimension $n$; we will denote: $\nabla$ its Levi-Civita connection (extended to all tensor bundles 
over $\mathcal{N}$); $\Riem$ its Riemann curvature tensor; $\inj \mathcal{N}$ its injectivity radius.
Unless otherwise stated, we assume the metric $g$ is fixed once and for all, 
and therefore we shall usually omit it from our notation. In this note, we will always deal with isometric Euclidean embeddings with positive reach  
(most often, the embedding provided by Theorem~\ref{thm:p}), and we will 
omit explicit reference to the embedding in our notation and terminology 
(for instance, we will speak of the ``reach of $\mathcal{N}$'', in place of 
the ``reach of $f(\mathcal{N})$ in $\R^\nu$'').
\end{notation}

\begin{definition} 
We say that $\mathcal{N}$ has \emph{uniformly polynomial growth} if there is 
a polynomial $p$ (called a {\em growth polynomial}) such that
\[
	\vol B_\mathcal{N}(x,R) \leq p(R),
\]
for any $x \in \mathcal{N}$ and any $R > 0$, 
where $B_\mathcal{N}(x,R)$ denotes the geodesic 
ball of centre $x$ and radius $R$ in $\mathcal{N}$. 
\end{definition}

\begin{definition}\label{def:bdd-geom} 
We say that $\mathcal{N}$ has 
{\em $1$-bounded geometry} if and only if:   
{(i)} $\inj \mathcal{N} > 0$; and {(ii)}  
	$\abs{\nabla \Riem} \leq C$ (in every normal chart).
\end{definition}

By Nash's embedding theorem, any Riemannian manifold $\mathcal{N}$ can be 
isometrically embedded into a Euclidean space $\R^\nu$, 
for some $\nu \in \N$.
Nash's construction provides the general upper-bound 
$\nu \leq \frac{1}{2} n(n+1)(3n+11)$. However, Nash's argument 
does not yield, in general, a uniform $\eps$-neighbourhood of 
$f(\mathcal{N})$ on which the nearest-point projection is well-defined. 
(Of course, if $\mathcal{N}$ is compact, then any isometric embedding 
of $\mathcal{N}$ into $\R^\nu$ has positive reach and smooth nearest-point 
projection --- see, e.g., {\cite[Section~4]{Federer}} or
\cite[Theorem~1 in Section~2.12.3]{Simon} for a fully detailed 
proof.)
Theorem~\ref{thm:p} stems from Petrunin's key idea that, under the assumptions that $\mathcal{N}$ 
has $1$-bounded geometry and uniformly polynomial growth, it is possible to improve on Nash's construction 
so as to obtain uniform bounds on the derivatives of $f$. In turn, these 
uniform bounds produce a uniform $\eps$-neighbourhood of 
$f(\mathcal{N})$ (see \cite[§4]{P} for more  
information about the strategy of the proof of Theorem~\ref{thm:p}).

To conclude this section, we mention that the proof of 
Theorem~\ref{thm:p} yields an upper-bound for $\nu$. Such an upper-bound 
is expressed in terms of $n$ and $\deg p$ only but making it explicit would 
require some further explanations on the constructions in \cite{P} that are not relevant 
to the purposes of the present note. 
We refer the interested reader to \cite[§13]{P}. 

\section{Extensions by averaging}\label{sec:averaging}

We recall here some terminology and notation from \cite{BVS} that will 
be used in Section~\ref{sec:proof}.

We write $x = (x',x_{m+1})$ for points of $\R^{m+1}_+$.
If $u : \R^m \to \R^\nu$ is any measurable map, we denote 
$V : \R^{m+1}_+ \to \R^\nu$ its \emph{extension by averaging}, defined 
by (see \cite[Eq.~(18)]{BVS} and \cite{Gagliardo})
\begin{equation}\label{eq:defV}
	V(x) := \frac{1}{x^m_{m+1}} \int_{\R^m} u(z) \varphi\left( \frac{x'-z}{x_{m+1}} \right) \,{\d}z 
	= \int_{\R^m} u(x'-x_{m+1} z) \varphi(z)\,{\d}z,
\end{equation}
where $\varphi \in C^\infty(\R^m,\R)$ is such that 
\[
	\int_{\R^m} \varphi\,{\d}z = 1, \quad \spt \varphi \subseteq \B^{m+1}_1, 
	\quad \norm{\varphi}_{L^\infty(\R^m)} \leq 1, \quad 
	\norm{{\D}\varphi}_{L^\infty(\R^m)} \leq 2.
\]
(The specific values of the  $L^\infty$-bounds above do not play any 
r\^{o}le but it is useful to fix such bounds once and for all.) 

\noindent The relevant properties of the extension by averaging are gathered in  
\cite[Proposition~2.2]{BVS}.

As in \cite[Eq.~(24)]{BVS}, for any given $\lambda \in (1,\infty)$, every 
$k \in \Z$, every $\tau \in (1,\lambda)$, and for any given $h \in \R^m$, we 
define the families of cubes
\[
\begin{aligned}
	& \mathcal{Q}_{\lambda, \tau, k, h} := 
	\left\{ \left.\tau \lambda^{-k}\left([0,1]^m + j + h\right) \, \right\vert \, j \in \Z^m \right\}, \\
	& \mathcal{Q}^+_{\lambda, \tau, k, h} := 
	\left\{ \left.\tau \lambda^{-k}\left([0,1]^m + \left(j, (\lambda-1)^{-1}\right) + h\right) \,\right\vert \, j \in \Z^m \right\}.
\end{aligned}
\]
We recall, in addition, the following proposition.
\begin{prop}[{\cite[Proposition~3.1]{BVS}}]\label{prop:BVS-3.1}
	Let $m \in \N \setminus \{0\}$. There exists constants $\eta \in (0,1)$ 
	and $C \in (0,\infty)$ depending only on $m$ such that for every 
	$\delta \in (0,\infty)$, for every $\lambda \in [2,\infty)$, for 
	every measurable function $u : \R^m \to \R^\nu$ and every set 
	$Y \subseteq \R^\nu$, if $V$ is an extension by averaging given 
	by~\eqref{eq:defV} and if $u \in Y$ almost everywhere in $\R^m$, then 
	\begin{multline}
	\int_1^\lambda \sum_{k \in \Z} \int_{[0,1]^m} 
		\sharp\left\{ \left. Q \in \mathcal{Q}^+_{\lambda,\tau,k,h} \,\right\vert \,\sup_{x \in \partial Q} \dist(V(x),\mathcal{N}) \geq \delta \right\} \,{\d}h\frac{\d{\tau}}{\tau} \\
		\leq \frac{C}{\delta^{m+1}} \iint \limits_{\substack{ (x,y) \in \R^m \times \R^m \\ d(u(x),u(y))\geq \delta}} \frac{(d(u(y),u(z)) - \eta \delta )^{m+1}_+}{\abs{y-z}^{2m}}\,{\d}y\,{d}z.
	\end{multline}
\end{prop}
Note that $Y$ can be {\em every} set in $\R^\nu$; in particular, it need not 
be compact, although
Proposition~\ref{prop:BVS-3.1} is applied in \cite{BVS} in the proof of 
Theorems~1.1, 1.2, and 1.3 with $Y = \mathcal{N}$ and $\mathcal{N}$ compact.

\begin{remark}
	As one can easily check, none of the preparatory results in \cite{BVS} before the proof of 
	\cite[Theorem~1.2]{BVS} requires the compactness of $\mathcal{N}$.
\end{remark}

\section{Proof of the theorems}\label{sec:proof}

In this section, we prove Theorem~\ref{thm:II}. As already explained in the 
Introduction, all the changes with respect to the proof of \cite[Theorem~1.2]{BVS} 
are confined to the beginning of the proof, therefore a very brief sketch explaining 
those will be enough. Moreover, following the arguments in \cite{BVS}, 
once those changes have been performed on the 
proof of \cite[Theorem~1.2]{BVS}, no others are needed in the proofs of 
\cite[Theorem~1.1 and Theorem~1.3]{BVS}, so that Theorem~\ref{thm:I} and 
Theorem~\ref{thm:III} follow automatically.

\begin{proof}[Proof of Theorem~\ref{thm:II}]
	Without loss of generality, we may assume $\mathcal{N}$ is non-compact. 
	Let $2 \delta_{\mathcal{N}} := \reach \mathcal{N}$ be the reach of 
	$\mathcal{N}$ in $\R^\nu$. By assumption, $2\delta_{\mathcal{N}} > 0$ 
	(moreover, we may assume  
	$\delta_{\mathcal{N}} < \infty$ --- see Example~\ref{ex:trivial} below)	
	and 
	then, by results in \cite{LeobacherSteinicke}, the nearest-point projection 
	$\Pi_{\mathcal{N}} : \mathcal{N} + \B^{\nu}_{2\delta_{\mathcal{N}}} \to \R^{\nu}$ 
	is smooth up to the boundary of the narrower neighbourhood 
	$\mathcal{N} + \B^\nu_{\delta_{\mathcal{N}}}$ of $\mathcal{N}$ in $\R^\nu$.
	
	Exactly as in the proof of \cite[Theorem~1.2]{BVS}, since 
	$u(x) \in \mathcal{N}$ for almost every $x \in \R^m$, 
	by Proposition~\ref{prop:BVS-3.1} (i.e., by \cite[Proposition~3.1]{BVS}), 
	we have 
	\begin{multline}\label{eq:avg-dist}
		\int_1^\lambda \sum_{k \in \Z} \int_{[0,1]^m} 
		\sharp\left\{ Q \in \mathcal{Q}^+_{\lambda,\tau,k,h} \vert \sup_{x \in \partial Q} \dist(V(x),\mathcal{N}) \geq \delta_{\mathcal{N}}/2 \right\} \,{\d}h\frac{\d{\tau}}{\tau} \\
		\leq C_1 \iint \limits_{\substack{ (x,y) \in \R^m \times \R^m \\ d(u(x),u(y))\geq \eta\delta_{\mathcal{N}} /2}} \frac{(d(u(y),u(z)) - \eta \delta_\mathcal{N} / 2 )^{m+1}_+}{\abs{y-z}^{2m}}\,{\d}y\,{d}z,
	\end{multline}
	where $\eta$ is the constant of Proposition~\ref{prop:BVS-3.1} (which 
	depends only on $m$) and 
	$C_1$ is the constant $C = C(m)$ of Proposition~\ref{prop:BVS-3.1} 
	divided by $\delta_\mathcal{N} /2$ (that is, $(\reach \mathcal{N}) / 4$), 
	so that $C_1$ depends only on $m$ and the reach of $\mathcal{N}$.
	
	Now, for almost every $y$ and $z$ in $\R^m$ and every $\eta \in (0,\infty)$, 
	there holds
	\begin{equation}\label{eq:thmII-compu1}
		(d(u(y),u(z)) - \eta \delta_\mathcal{N} / 2 )_+ 
		\leq d(u(y),u(z)).
	\end{equation}
	If, in addition, $\diam_{\mathcal{N}}{u(\S^m)} \leq L$ then for almost every 
	$y$ and $z$ in $\R^m$, we have 
	\[
		(d(u(y),u(z)) - \eta \delta_\mathcal{N} / 2 )^{m+1}_+ 
		\leq (d(u(y),u(z)))^{m+1}
		\leq L^{m+1},
	\]
	so that 
	\[
		\iint \limits_{\substack{ (x,y) \in \R^m \times \R^m \\ d(u(x),u(y))\geq \delta}} \frac{(d(u(y),u(z)) - \eta \delta_\mathcal{N} / 2 )^{m+1}_+}{\abs{y-z}^{2m}}\,{\d}y\,{d}z \leq L^{m+1} \iint \limits_{\substack{ (x,y) \in \R^m \times \R^m \\ d(u(x),u(y))\geq \delta}}\frac{1}{\abs{y-z}^{2m}} \,{\d}y\,{\d}z, 
	\]
	and the rest of the proof follows {\em verbatim} as in \cite{BVS}, setting 
	$\delta := \eta \delta_{\mathcal{N}}/2$ and choosing 
	\[
		\lambda := 1 + \exp\Biggl( 2 C_1 L^{m+1} \iint \limits_{\substack{ (x,y) \in \R^m \times \R^m \\ d(u(x),u(y))\geq \delta}} \frac{1}{\abs{y-z}^{2m}} \,{\d}y\,{\d}z \Biggr).
	\]
	In the end, this yields~\eqref{eq:II-2}, with $B' = 2 C_1 L^{m+1}$ 
	depending only on $m$, $L$, and $\reach{\mathcal{N}}$, and $A$ exactly 
	as in \cite{BVS} depending only on $m$.
	
	In the general case, we set once again $\delta := \eta \delta_{\mathcal{N}}/2$ but 
	(keeping~\eqref{eq:thmII-compu1} in mind)
	\[
		\lambda := 1 + \exp\left( 2 C_1 \iint \limits_{\R^m \times \R^m} \frac{(d(u(y),u(z)))^{m+1}}{\abs{y-z}^{2m}} \,{\d}y\,{\d}z \right).
	\]
	The integral in the definition of $\lambda$ is finite because 
	it is exactly the Gagliardo energy of $u$. Once again, the rest of the 
	proof applies {\em verbatim}, leading in the end to~\eqref{eq:II-1}, 
	with $B$ depending only on $m$ and $\reach{\mathcal{N}}$ and $A$ exactly 
	as in \cite{BVS}, and therefore depending only on $m$.
\end{proof}

\section{Examples}\label{sec:examples}

We conclude this note by listing some examples of target manifolds to which 
Theorem~\ref{thm:I}, Theorem~\ref{thm:II}, and Theorem~\ref{thm:III} apply 
and that are not covered by the results in \cite{BVS}. 

\begin{example}\label{ex:trivial}
	Trivial examples are provided by $\mathcal{N} = \R^\nu$ or any closed convex set in 
	$\R^\nu$.
	By \cite[4.8(8)]{Federer}, a set $\mathcal{C} \subset \R^\nu$ is closed and convex 
	if and only if $\reach{\mathcal{C}} = +\infty$. 
	If $u : \R^m \to \mathcal{C}$, then we can construct 
	its extension by averaging  $V :\R^{m+1}_+ \to \R^\nu$	
	as in~\eqref{eq:defV}, and then compose $V$ with a Lipschitz retraction 
	into $\mathcal{C}$. By the chain rule 
	of Sobolev mappings, we obtain  
	a map $U \in W^{1,m+1}(\R^{m+1},\,\mathcal{C})$ satisfying  
	a \emph{linear} version of~\eqref{eq:II-1}, in which 
	the Gagliardo seminorm appears linearly on the right-hand side 
	(in fact, by the classical 
	trace theory by E.~Gagliardo \cite{Gagliardo}, it also 
	satisfies the linear, strong-type control 
	$
		\int_{\R^{m+1}_+} \abs{\D{U}}^{m+1} \leq C \iint \limits_{\R^m \times \R^m} \frac{\abs{u(x)-u(y)}^{m+1}}{\abs{x-y}^{2m}}\,{\d}x\,{\d}y
	$, where $C=C(m)$ is a positive constant depending only on $m$). 
	As an immediate consequence of the absence of bad cubes, 
	we also obtain linear versions of~\eqref{eq:I-1} and~\eqref{eq:III-1}.
\end{example}

\begin{example}\label{ex:warped}
	A first non-trivial example is provided by warped-products of the type 
	$\mathcal{N} := \R \times_f \mathcal{M}$, where $\mathcal{M}$ is any 
	compact, connected, smooth Riemannian manifold and $f : \R \to \R$ is a 
	smooth warping function, satisfying $0 < a \leq f \leq b$ for positive 
	real numbers $a$ and $b$, and 
	with bounded derivative up to order 3 (so to ensure $\mathcal{N}$ has 
	$1$-bounded geometry). 
	Since every compact manifold has $1$-bounded geometry and uniformly 
	polynomial growth \cite[Section~4]{Federer}, it follows that 
	$\mathcal{N}$ has uniformly polynomial growth.
	Then, Theorem~\ref{thm:p} yields  
	that $\mathcal{N}$ has a tubed embedding in $\R^\nu$ and our theorems can 
	be applied.
\end{example}

\begin{example}\label{ex:covering}
	Let $\mathcal{M}$ be a compact Riemannian manifold and let $\mathcal{N}$ 
	be its universal Riemannian covering. 
	Assume that $\pi_1(\mathcal{M})$ has polynomial growth. 
	Then, $\mathcal{N}$ has uniformly polynomial growth 
	(and viceversa; see, e.g., \cite[Corollary, p.~57]{Gromov}). 
	Moreover, since $\mathcal{M}$ 
	is compact and $\mathcal{N}$ is its universal covering, $\mathcal{N}$ has 
	positive injectivity radius\footnote{In fact, 
	$\inj \mathcal{M} \leq \inj \mathcal{N}$, as for any 
	$p \in \mathcal{M}$ there holds $\inj_p \mathcal{M} \leq \inj_{\hat{p}} \mathcal{N}$ where $\hat{p}$ is any lift of $p$ to $\mathcal{N}$.} 
	and $1$-bounded geometry. 
	Thus, by 
	Theorem~\ref{thm:p}, $\mathcal{N}$ has an isometric Euclidean embedding 
	with positive 
	reach, and Theorems~\ref{thm:I},~\ref{thm:II},~\ref{thm:III} hold for 
	maps with values into $\mathcal{N}$. 

	\noindent Striking examples of compact manifolds whose fundamental group 
	has polynomial growth are:
	\begin{itemize}
		\item Every compact manifold with non-negative Ricci curvature 
		tensor (\cite[Theorem~1]{Milnor});
		\item Every compact manifold $\mathcal{M}$ such that 
		$\pi_1(\mathcal{M})$ is Abelian.
	\end{itemize}
	Several other relevant examples of Riemannian manifolds with 
	uniformly polynomial growth are listed in \cite[p.~57]{Gromov}.	  
\end{example}

\begin{example}
	A Riemannian metric $h$ on $\R^n$ is called 
	\emph{asymptotically Euclidean} if the set 
	$K := \spt\{h - g_{\rm eucl}\}$ 
	is compact. Since $K$ is compact, $h$ and $g_{\rm eucl}$ are comparable 
	and $(\R^n,h)$ has positive injectivity radius (in particular, it is 
	complete).
	Moreover, 
	$(\R^n,h)$ is clearly of $1$-bounded geometry (for any $k \in \N$), hence 
	Theorem~\ref{thm:p} applies.
\end{example}

\begin{example}
	A Riemannian manifold $(\mathcal{N},g)$ is called 
	\emph{conical at infinity} if there exists a compact Riemannian manifold 
	$(\mathcal{M},h_0)$ of dimension $n-1$, a compact set $K$ in 
	$\mathcal{N}$ and a diffeomorphism from $\mathcal{N}\setminus K$ 
	to $[r_0,\infty)\times \mathcal{M}$, such that, outside $K$,
	\[
		g = {\d}r^2 + r^2 h_0,
	\]
	so that $g$ is conical outside $K$. 
	If $(\mathcal{N},g)$ is conical at infinity, then it is complete and it 
	has $1$-bounded geometry (see, e.g., \cite{BGIM}). 
	Moreover, it is clear that manifolds conical at infinity have uniformly 
	polynomial growth. Thus, Theorem~\ref{thm:p} applies to them. 
	(The reference \cite{BGIM} contains several other interesting, explicit 
	examples of manifolds with $1$-bounded geometry.) 
\end{example}

\begin{example}
	\emph{Asymptotically Locally Euclidean (ALE) manifolds} 
	(with a single chart at infinity, see, e.g., \cite[Definition~4.13]{AFM}) 
	have Euclidean (hence, polynomial)
	volume growth and, by tuning the decay rate, one can obtain ALE manifolds 
	with $1$-bounded geometry, to which Theorem~\ref{thm:p} applies.
\end{example}

The list above is merely illustrative and by no means exhaustive.
We conclude our list with two examples of noncompact manifolds which, 
according to Theorem~\ref{thm:p}, do not admit tubed Euclidean embeddings.

\begin{example}
As already remarked in \cite{P}, 
the hyperbolic space {\em cannot} have a tubed Euclidean embedding because 
its volume growth rate is exponential. Note that the hyperbolic space 
satisfies all the other assumptions of Theorem~\ref{thm:p}.
\end{example}

\begin{example}
	By a theorem of Milnor \cite[Theorem~2]{Milnor}, if $\mathcal{M}$ is any 
	compact manifold with all sectional curvature strictly less than zero, 
	then $\pi_1(\mathcal{M})$ has exponential growth. 
	By Gromov's result mentioned in Example~\ref{ex:covering}, 
	the universal Riemannian covering $\mathcal{N}$ of $\mathcal{M}$ 
	{\em cannot} have uniformly polynomial growth. Therefore, it cannot have 
	a tubed Euclidean embedding. Again, $\mathcal{N}$ satisfies all the 
	other assumptions of Theorem~\ref{thm:p}.
\end{example} 

\paragraph{Acknowledgements} 
The author is a member of GNAMPA-INdAM, partially supported by the 
GNAMPA project \textsc{CUP\_E53C25002010001}. 
The author would like to warmly thank  
Bohdan Bulanyi for bringing to his attention a mistake in the first draft of this paper.

\paragraph{Data availability}
No data was used for the research described in the article.

\paragraph{Conflict of interest} 
The author has no Conflict of interest to declare that are relevant to the content of this
article.

\bibliographystyle{plain}
\bibliography{biblio}

\Addresses

\end{document}